\documentclass[12pt]{article}
\usepackage[T1]{fontenc}
\usepackage{lmodern,amsmath,amsthm,amsfonts,amssymb,graphicx,float,microtype,thmtools,underscore,mathtools,xurl,multirow}
\usepackage[usenames,dvipsnames,svgnames,table]{xcolor}
\usepackage[shortlabels]{enumitem}
\setlist[itemize]{topsep=0ex,itemsep=0ex,parsep=0ex}
\setlist[enumerate]{topsep=0ex,itemsep=0ex,parsep=0ex}
\usepackage{pgf,tikz,ifthen,arrayjobx}
\usetikzlibrary{arrows}
\usetikzlibrary{calc}
\usepackage[unicode=true]{hyperref}
\hypersetup{ 
colorlinks,
breaklinks=true,
linkcolor={blue!60!black},
citecolor={black},
urlcolor={blue!60!black},
pdftitle={Notes on Graph Covers}}
\usepackage[capitalise, compress, nameinlink, noabbrev]{cleveref}
\crefname{lem}{Lemma}{Lemmas}
\crefname{thm}{Theorem}{Theorems}
\crefname{prop}{Proposition}{Propositions}
\crefname{cor}{Corollary}{Corollaries}
\usepackage[longnamesfirst,numbers,sort&compress]{natbib}
\makeatletter
\def\NAT@spacechar{~}
\makeatother
\usepackage[margin=30mm]{geometry}
\renewcommand{\baselinestretch}{1.15}
\allowdisplaybreaks
\renewcommand{\epsilon}{\varepsilon}

\renewcommand{\geq}{\geqslant}
\renewcommand{\leq}{\leqslant}
\renewcommand{\thefootnote}{\fnsymbol{footnote}}
\theoremstyle{plain}
\newtheorem{thm}{Theorem}

\newtheorem{prop}[thm]{Proposition}

\crefname{obs}{Observation}{Observations}
\newtheorem*{lem*}{Lemma}
\theoremstyle{definition}

\newtheorem*{conj*}{Conjecture}
\newtheorem{remark}{Remark}
\newtheorem{fact}{Fact}
\theoremstyle{problem}
\begin{document}
\title{\bf\boldmath\fontsize{18pt}{18pt}\selectfont A Note on Inequalities for Three Domination Parameters%Involving $\gamma, \gamma_t$, and $\gamma_c$
}

\author{%
Dickson Y. B. Annor\,\footnotemark[5] \qquad\\
%Department of Mathematical and Physical Sciences,\\
%La Trobe University, Bendigo, Australia
%Michael S.~Payne\,\footnotemark[5] \qquad 
%David~R.~Wood\,\footnotemark[2]
}

\date{}

\maketitle

\begin{abstract}
In this short paper, we establish relations between the domination number $\gamma$, the total domination number $\gamma_t$, and the connected domination number $\gamma_c$ of a graph. In particular, we prove upper and lower bounds for $\gamma_t$ in terms of $\gamma$ and $\gamma_c$. %Moreover, we propose the following conjecture: for every connected isolated-free graph $G$,
%\begin{equation*}\label{eq:low}
%  \gamma_t(G) \geq  \left \lfloor \frac{3\gamma(G) +2\gamma_c(G)}{6}\right\rfloor.
%\end{equation*}
%As evidence to support the conjecture, we prove that the conjecture holds when $\gamma_t(G) = \gamma_c(G)$ and also, when $\gamma_t(G) = \gamma_c(G) -1$.
\end{abstract}



\footnotetext[5]{Department of Mathematical and Physical Sciences, La Trobe University, Bendigo, Australia 
(\texttt{\ d.annor@latrobe.edu.au}). Research of Annor supported by a La Trobe Graduate Research Scholarship. }

\renewcommand{\thefootnote}{\arabic{footnote}}

%\maketitle

\section{Introduction}
In this paper, we deal with finite undirected and simple graphs. For a graph $G$, let $V(G)$ and $E(G)$ respectively denote the vertex set and the edge set of $G$.

Let $S$ be a subset of the vertex set of a graph $G$. $S$  is called a \emph{dominating set} of $G$ if
every vertex not in $S$ is adjacent to at least one vertex in $S$. 
The \emph{domination number} of $F$, denoted by $\gamma(G)$, is the minimum cardinality of a dominating set of $F$. A \emph{total dominating set} of $G$ is a set
$S$ such that every vertex is adjacent to some vertex in $S$. Note that by the definition of a total dominating set, it is evident that a graph $G$ admits a total dominating set if and only if $G$ has no isolated vertices.
The \emph{total domination number} of $G$, denoted by $\gamma_t(G)$, is the minimum cardinality of a total dominating set of $G$. A \emph{connected dominating set} $S$ of $G$ is a set such that every vertex not in $S$ is adjacent with at least one vertex in $S$ and the subgraph induced by $S$ is connected. The \emph{connected domination number} $\gamma_c(G)$ is the minimum cardinality of a connected dominating set of $G$. Note that a graph admits a connected dominating set if and only if the graph is connected.

The theory of domination in graphs is one of the main research areas in graph theory. %extensively studied in the literature. 
In 1979, Bollob\'as and Cockayne \cite{bollobas1979graph} observed that the total domination number of an isolate-free graph is
squeezed between the domination number and twice the domination number.

\begin{thm}[\cite{bollobas1979graph}]\label{them:gt2g}
If $G$ is an isolated-free graph, then $\gamma(G) \leq \gamma_t(G) \leq 2\gamma(G)$.   
\end{thm}

It is obvious that a connected dominating set is a dominating set. Sampathkumar and
Walikar \cite{sampathkumar1979connected} proved the following result.

\begin{thm}[\cite{sampathkumar1979connected}]\label{thm:condom}
If $G$ is a connected graph, then
$\gamma(G) \leq \gamma_c(G) \leq 3 \gamma(G) - 2$.   
\end{thm}

The following easy fact was observed in \cite{henning2022bounds}.

\begin{fact}\label{fact:fun}
If $G$ is a connected isolated-free graph and $\gamma(G)> 1$, then $\gamma(G)\leq \gamma_t(G) \leq \gamma_c(G)$.    
\end{fact}

A natural question is: what relationships exist between $\gamma, \gamma_t$ and $\gamma_c$?

To our knowledge, little or nothing is known about the relations between $\gamma,\gamma_t$ and $\gamma_c$, (and moreover, any three domination parameters in general). %, apart from Fact~\ref{fact:fun}. the domination number, the total domination number, and the connected domination number of a graph.
This motivates us to study the relationships between $\gamma(G), \gamma_t(G)$ and $\gamma_c(G)$ when $G$ is a connected isolated-free graph.

For an excellent treatment of fundamentals of domination in graphs and recent topics on bounds for these three domination parameters, we refer the reader to the following books \cite{haynes2023domination,haynes2020topics,haynes2013fundamentals}.

\section{Main results} 
In this section, we demonstrate some relationships and show that these results are best possible. %is devoted to establishing relations concerning $\gamma(G), \gamma_t(G)$ and $\gamma_c(G)$ of a connected isolated-free graph $G$.

%From Theorems~\ref{them:gt2g} and Fact~\ref{fact:fun}, we can easily deduce the following.

\begin{thm}\label{prop:con-sum}
 If $G$ is a connected isolated-free graph and $\gamma(G)> 1$, then
 \begin{equation*}
 \gamma_t(G) \leq \gamma(G) + \frac{1}{2}\gamma_c(G).    
 \end{equation*}
 Moreover, this bound is tight.
\end{thm}

\begin{proof}
 Let $G$ be a connected isolated-free graph with $\gamma(G)>1$.  From Theorem~\ref{them:gt2g} and Fact~\ref{fact:fun}, we have $\gamma_t(G) \leq 2\gamma(G)$ and $\gamma_t(G) \leq \gamma_c(G)$, respectively. By adding these two inequalities, we get $2\gamma_t(G) \leq 2\gamma(G) + \gamma_c(G)$ and the result follows immediately. In Figure~\ref{fig:1tight}, it can easily be checked that $\gamma(H) = 5$ and $\gamma_t(H) = \gamma_c(H)=10$, which proves that the bound is tight.
\end{proof}

\begin{figure}
    \centering
\begin{tikzpicture}[scale=0.5]
 \draw[ thick] (0,0) -- (2,0); 
\draw[ thick] (2,0) -- (4,0); 
\draw[ thick] (0,0) -- (0,-1.5); 
\draw[ thick] (0,-3) -- (0,-1.5);  
\draw[ thick] (2,0) -- (2,-1.5);  
\draw[ thick] (2,-1.5) -- (2,-3,0); 
\draw[ thick] (4,0) -- (4,-1.5); 
\draw[ thick] (4,-1.5) -- (4,-3);  
\draw[ thick] (2,0) -- (3,1); 
\draw[ thick] (3,1) -- (5,1);
\draw[ thick] (5,1) -- (7,1); 
\draw[ thick] (2,0) -- (1,1); 
\draw[ thick] (1,1) -- (-1,1);  
\draw[ thick] (-1,1) -- (-3,1); 
\filldraw[black] (0,0) circle (3pt);
\filldraw[black] (2,0) circle (3pt);
\filldraw[black] (4,0) circle (3pt);
\filldraw[black] (0,-1.5) circle (3pt);
\filldraw[black] (0,-3) circle (3pt);
\filldraw[black] (2,-1.5) circle (3pt);
\filldraw[black] (2,-3) circle (3pt);
\filldraw[black] (4,-1.5) circle (3pt);
\filldraw[black] (4,-3) circle (3pt);
\filldraw[black] (1,1) circle (3pt);
\filldraw[black] (-1,1) circle (3pt);
\filldraw[black] (-3,1) circle (3pt);
\filldraw[black] (3,1) circle (3pt);
\filldraw[black] (5,1) circle (3pt);
\filldraw[black] (7,1) circle (3pt);
\end{tikzpicture}
    \caption{Graph $H$.}
    \label{fig:1tight}
\end{figure}

Similarly, we can deduce the following result from Theorems~\ref{them:gt2g} and~\ref{thm:condom}.

\begin{thm}\label{Prop:con-dif}
 If $G$ is a connected isolated-free graph, then
 \begin{equation*}
 \gamma_t(G) \leq 5\gamma(G) - \gamma_c(G)-2.    
 \end{equation*}
 Moreover, this bound is tight.
\end{thm}

Equality in Theorem~\ref{Prop:con-dif} is achieved for complete graphs. In fact, the following generalisation of Theorem~\ref{prop:con-sum} holds for all connected isolated-free graphs.

\begin{thm}\label{thm:allcon}
 For any connected isolated-free graph $G$,
 \begin{equation*}
  \gamma_t(G) \leq  \left \lceil \frac{2(\gamma(G) +\gamma_c(G))}{3}\right\rceil.
\end{equation*}  
Moreover, this bound is tight.
\end{thm}

\begin{proof}
Let $G$ be a connected isolated-free graph. If $\gamma(G) > 1$, then it follows from Fact~\ref{fact:fun} that $2\gamma_t(G) \leq 2\gamma_c(G)$. So, by Theorem~\ref{them:gt2g}, we have \begin{equation*}
 2\gamma_t(G) + \gamma_t(G) \leq 2\gamma(G) + 2\gamma_c(G) \Rightarrow \gamma_t(G) \leq \frac{2(\gamma(G) + \gamma_c(G))}{3}.  
\end{equation*} 
On the other hand, if $\gamma(G) = 1$, then $G$ has a dominant vertex. That is, $G$ has a vertex that dominates every other vertex of $G$. In that case, $\gamma_t(G) = 2$. So, $\lceil \frac{2 \times 2}{3}\rceil = 2$. Furthermore, this proves that the bound is tight.
\end{proof}

%A natural question that arises is: can we find a lower bound for $\gamma_t(G)$ in terms of $\gamma(G)$ and $\gamma_c(G)$? 

Now we turn our attention to the lower bounds. First, we point out the following lower bound,  which is a consequence of Theorems~\ref{them:gt2g} and~\ref{thm:condom}.

\begin{prop}\label{prop:notstrong}
 If $G$ is a connected isolated-free graph, then
 \begin{equation*}
 \gamma_t(G) \geq 2\gamma(G) - \gamma_c(G).    
 \end{equation*}
 Moreover, this bound is tight.
\end{prop}

\begin{figure}[ht]
    \centering
    \begin{tikzpicture}[scale=0.3]
 \draw[ thick ] (0,0)--(8,0);
 \draw[ thick ] (0,0)--(4,8);
 \draw[ thick ] (8,0)--(4,8);
 \draw[ thick ] (2,4)--(6,4);
\draw[ thick ] (4,0)--(6,4);
 \draw[ thick ] (4,0)--(2,4);
 \filldraw[black] (0,0) circle (6pt);
 \filldraw[black] (4,0) circle (6pt);
 \filldraw[black] (8,0) circle (6pt);
 \filldraw[black] (2,4) circle (6pt);
% \filldraw[black] (8,8) circle (6pt);
 \filldraw[black] (4,8) circle (6pt);
 \filldraw[black] (6,4) circle (6pt);
    \end{tikzpicture}
    \caption{Graph $G'$.}
    \label{fig:K2221}
\end{figure}

%As part of the answer to the above question,
\begin{remark}
%We remark that
Equality is achieved for the graph $G'$ in Figure~\ref{fig:K2221}. %Moreover, the bound in Proposition~\ref{prop:notstrong} is not strong, as can be observed from the upper bound in Theorem~\ref{thm:condom}. We believe that a stronger lower bound has the same form as Theorem~\ref{thm:allcon}. So, we propose the following conjecture.   
\end{remark}

Finally, we end this section by proving the following bound.% weakening of Conjecture~\ref{conj:lower}.

\begin{thm}\label{thm:lower}
For every connected isolated-free graph $G$,
\begin{equation*}\label{eq:wlow}
  \gamma_t(G) \geq  \left \lfloor \frac{3\gamma(G) +\gamma_c(G)}{6}\right\rfloor.
\end{equation*}
\end{thm}

\begin{proof}
Suppose not. Then $\gamma_t(G) <  \left \lfloor \frac{3\gamma(G) +\gamma_c(G)}{6}\right\rfloor$, which implies that $\gamma_t(G) <   \frac{3\gamma(G) +\gamma_c(G)}{6}$  because $\gamma_t(G) \in \mathbb{N}$. Applying Theorem~\ref{thm:condom}, we get
\begin{equation*}
  6\gamma_t(G) < 3\gamma(G) + \gamma_c(G) \leq 3\gamma(G) + 3\gamma(G) -2 = 6\gamma(G)-2.   
\end{equation*}
Once again, Theorem~\ref{them:gt2g} implies that $6\gamma(G) \leq 6\gamma_t(G) < 6\gamma(G)-2$; a contradiction.    
\end{proof}

\section{Conclusion}
In this paper, we established relationships between $\gamma(G), \gamma_t(G)$ and $\gamma_c(G)$ for a connected isolated-free graph $G$. %how three domination parameters behave under graph covers. %%the connections between three domination parameters and graph covers.
%We hope that another major direction of research on domination in graphs is suggested by the results and the conjecture presented in this paper. %%Furthermore, our results suggest that the problem of characterising graphs that achieve equality in any of the upper bounds would be interesting. 

%$x_{i,1}x_{i,2}x_{i,3}x_{i,4}$

\bibliographystyle{plain}
\bibliography{references}

\end{document}